\documentclass[a4paper,11pt]{article}

\usepackage{amsmath}
\usepackage{amsfonts}
\usepackage{amssymb}
\usepackage{amsthm}
\usepackage{mathtools}
\usepackage{booktabs}
\usepackage{indentfirst}
\usepackage{hyperref}
\usepackage{enumerate}
\usepackage{tikz}

\newtheorem{theorem}{Theorem}[section]
\newtheorem{lemma}[theorem]{Lemma}
\newtheorem{corollary}[theorem]{Corollary}
\newtheorem{question}[theorem]{Question}

\theoremstyle{definition}
\newtheorem{definition}[theorem]{Definition}
\newtheorem{prop}[theorem]{Proposition}
\newtheorem{example}[theorem]{Example}

\numberwithin{equation}{section}

\begin{document}

\title{Asymptotic Counting of Singular Fibers of Genus $g$}
\author{Qiyang Kong\quad Xiao-Lei Liu\quad Jiayao Wang}
\date{}

\maketitle

\begin{abstract}
Motivated by the effective form of Shafarevich's conjecture, we prove that the number of classes of singular fibers of genus $g$ has the exact asymptotic order $g^{\Theta(g)}$.
\end{abstract}



\section{Introduction}
\subsection{Shafarevich's conjecture}
Throughout this paper, we work over the field of complex numbers $\mathbb{C}$. Let $X$ be a smooth projective surface and $C$ a smooth projective curve. A fibration of $X$ is a surjective morphism $f: X\to C$ with connected fibers. 
All but finitely many fibers of $f$ are smooth, and a fiber is called singular if it is not smooth. 
Let $F$ be a general fiber of $f$. The genus of $F$, denoted by $g$, is called the genus of the fibration $f$ and is independent of the choice of $F$.  

Fix a smooth complex projective base curve $C$, and a finite subset $\Delta \subset C$. In his 1962 ICM lecture,  Shafarevich (\cite{shafarevich-ICM-1962}) conjectured that for a fixed pair $(C, \Delta)$ and an integer $g \geq 2$, there are only finitely many  isomorphism classes of non-isotrivial minimal families $f:X \to C$ of curves of genus $g$ with $X$ smooth such that $f:X \setminus f^{-1}(\Delta) \to C \setminus \Delta$ is a smooth family. This conjecture was proved by Parshin (\cite{parsin-curves}), case $\Delta= \emptyset$) and  Arakelov (\cite{arakelov}).

A natural further question is whether this finiteness result can be made effective. In this direction, Caporaso (\cite{caporaso2002certain}) observed that the number of such families can be bounded uniformly in terms of  $(g,q,s)$, where $q=g(C)$ and $s=|\Delta|$. In 2004, Heier (\cite{heier2004uniformly}) further obtained an explicit effective bound for this number, bounded asymptotically by $\exp(\exp(O(g(gq+s)^2)))$. 

Our goal is to obtain a sharper effective upper bound for the number of families appearing in Shafarevich's conjecture. As a first step toward this goal,  we are naturally led to the following  counting problem:
\begin{question}
	For a fixed integer $g\geq 2$, how many  types of singular fibers of genus $g$?
\end{question}

\subsection{Classification of singualr fibers}
The classification of singular fibers is a central problem in the study of fibrations of surfaces.
Kodaira (\cite{kodaira1963compact}) classified singular fibers of genus 1, and the genus 2 case was studied by Ogg (\cite{ogg1966pencils}) and completely classified by Namikawa and Ueno (\cite{namikawa1973complete}). For genus 3, numerical and topological classifications were obtained by Uematsu (\cite{uematsu1999numerical}) and by Ashikaga and Ishizaka (\cite{ashikaga2002classification}), respectively.

The classification of singular fibers plays an important role in the compactification of moduli spaces of curves (\cite{DeligneMumford1969,artin1971degenerate,liu1993courbes}), related algorithms (\cite{tate2006algorithm,liu1994modeles}) and arithmetic geometry (\cite{liu1994conducteur}). In addition, the classification of singular fibers has applications in high energy physics (\cite{xie2023pseudoperiodicmapclassificationtheories,lawrie2018theories}).

Since singular fibers of low genus have been completely classified, it is natural to extend this classification to an arbitrary genus.

 Recently, Dokchitser (\cite{dokchitser2025classificationreductiontypescurves}) studied reduction types of curves of arbitrary genus. A reduction type
$R=(S;m;p;g;\cdot)$
records the components $S$ of  the minimal regular normal crossing model of a special fiber, together with their multiplicities $m$, arithmetic genus $p$, geometric genus $g$ and intersection pairing. 
Reduction types that vary in the depth parameter $n$ of their inner chains are grouped into the same reduction family.
If $N(g)$ denotes the number of reduction families of genus $g$, then Dokchitser established the asymptotic lower bound
$ N(g)\geq g^{g-o(g)}
\text{ as }g\to\infty.$

Our classification proceeds in two steps. First,
the topological classification of degenerations is formulated in terms of topological monodromy and is determined by the triple
$(G, \sigma, S_f),$
where $G$ is a stable weighted graph of genus $g$, $\sigma$ is an
automorphism of $G$, and $S_f$ is the minimal generalized quotient (\cite{matsumoto2011pseudo}).
We classify the generalized quotient spaces, which may vary with the integer parameter $K$ determined by the screw number and the boundary valencies, into finitely many unified generalized quotient types. 
Therefore, we obtain the following upper bound.
\begin{theorem}
	For any $g\geq 2$, let $|\overline{\mathcal D}_g|$ denote the number of unified types of degenerations of genus $g$. Then 
	\[
	|\overline{\mathcal D}_g|
	\leq g^{O(g)}
	\text{ as } g\to\infty.
	\]
\end{theorem}

However, the homeomorphism class of  normally minimal singular fibers depends only on the minimal generalized quotient $S_f$ (\cite{matsumoto2011pseudo}). Consequently, we have the following corollary.
\begin{corollary}
	For any $g\geq 2$, let $|\overline{\mathcal F}_g|$ denote the number of unified types of singular fibers of genus $g$. Then
	\[
	|\overline{\mathcal F}_g|\leq |\overline{\mathcal D}_g|
	\leq  g^{O(g)}
	\text{ as } g\to\infty.
	\]
\end{corollary}

In fact, the two classifications of singular fibers coincide. Dokchitser's reduction types are equivalent to the dual graphs of normally minimal singular fibers. Moreover, the integer parameter $K$ in the generalized quotient space coincides with the inner chain depth $n$ of the reduction type. Consequently, by combining Dokchitser's lower bound with our upper bound, we deduce a more precise asymptotic order for the number of classes of singular fibers.

\begin{theorem}
	As $g\to\infty$, the numbers of reduction families $N(g)$, unified types of singular fibers $|\overline{\mathcal F}_g|$, and unified types of degenerations $|\overline{\mathcal D}_g|$ all have the same asymptotic order:
	\[
	N(g) = |\overline{\mathcal F}_g| = g^{\Theta(g)},
	\]
	and 
	\[
	|\overline{\mathcal D}_g| =g^{\Theta(g)} .
	\]
\end{theorem}

For  the effective Shafarevich's conjecture, if a fibration of genus $g$ is uniquely determined by the classification of its singular fibers, then the number of  fibrations of genus $g$ is bounded above by
$ g^{\Theta(sg)}.$

We make essential use of the theory of pseudo-periodic maps developed by Matsumoto and Montesinos (\cite{matsumoto2011pseudo}). Our work builds on the  classification procedure established by Ashikaga and Ishizaka in genus $3$ (\cite{ashikaga2002classification}). Their excellent contributions provide an important foundation for our study of the number of singular fibers of genus $g$.

\section{Periodic map}
\subsection{The number of the conjugacy classes  of  periodic maps}\label{upper}
We follow Ashikaga and Ishizaka (\cite{ashikaga2002classification}) for the basic concepts, notation, and relevant results on periodic maps used in this section.

By a surface $\Sigma$, we mean an oriented connected real 2-dimensional manifold with or without boundary. When we emphasize its complex structure, we call $\Sigma$ a Riemann surface.

A \emph{periodic map} on an oriented surface $\Sigma$ is an orientation-preserving homeomorphism $f:\Sigma\to\Sigma$ such that $f^{n}=\mathrm{id}_\Sigma$ for some integer $n\ge 2$. The smallest such $n$ is called the \emph{order} (or \emph{period}) of $f$.

Let $f:\Sigma \to \Sigma$ be a periodic map of order $n$, and let $P$ be a point on the surface $\Sigma$. Let $\alpha(P)$ be the smallest positive integer such that $f^k(P) \ne P$ for $1 \le k < \alpha(P)$, and $f^{\alpha(P)}(P) = P$. The point $P$ is called a \emph{simple point} if  $\alpha(P) = n$. Otherwise, it is called a \emph{multiple point}.

Let $C\subset\Sigma$ be an oriented simple closed curve. Let $m$ be the smallest positive integer such that $f^m(C)=C$ as a set and $f^{m}$ preserves the orientation of $C$. The restriction of $f^m$ to $C$ is a periodic map of order, say, $\lambda \ge 1$. Note that $n = m\lambda$.
Let $Q$ be any point on $C$, and suppose that the images of $Q$ under the iteration of $f^m$ are ordered as $(Q,\, f^{m\sigma}(Q),\, f^{2m\sigma}(Q),\dots,\, f^{(\lambda-1)m\sigma}(Q))$ viewed in the direction of $C$, where $\sigma$ is an integer with $0 \le \sigma \le \lambda - 1$ and $\gcd(\sigma,\lambda) = 1$. Let $\delta$ be the integer satisfying 
\begin{equation*}
	\sigma \delta \equiv 1 \pmod{\lambda},~~0 \leq \delta \leq \lambda -1.
\end{equation*}
Under suitable coordinates, $C$ can be viewed as an oriented unit circle, and the action of $f^m$ on $C$ can be interpreted as a rotation by $2\pi\delta/\lambda$. We call the triple $(m,\lambda,\sigma)$ the \emph{valency} of $C$ with respect to $f$.

Let $P\in\Sigma$ be a multiple point of $f$. Choose a sufficiently small disk neighborhood $D_{P}$ of $P$, and orient $\partial D_P$ from the outside of $D_P$. The \emph{valency} of $P$ is defined to be the valency of the oriented simple closed curve $\partial D_{P}$ with respect to $f$.

The periodic map $f$ induces a cyclic cover $\Pi:\Sigma \to \Sigma'=\Sigma /\langle f \rangle$ of degree $n$. The multiple points of $f$ coincide with the ramification points of $\Pi$. We define the \emph{valency of a branch point} $P$ on $\Sigma'$ as the valency of a ramification point in $\Pi^{-1}(P)$.

By Nielsen's theorem, the conjugacy classes of periodic maps are entirely determined by the valencies of their multiple points (\cite{nielsen1937structur}).

\begin{prop}\label{important prop}
	Let $\Sigma_g$ be a closed surface of genus $g $. Denote by
	$\Pi:\Sigma_g\to \Sigma'=\Sigma_g/\langle f\rangle$  the $n$-fold cyclic covering induced by a periodic map $f:\Sigma_g\to \Sigma_g$ of order $n$. Let  $\lambda_1,\ldots, \lambda_l$ be the ramification indices of $\Pi$ and $(n/\lambda_i, \lambda_i, \sigma_i)(1\leq i \leq l)$ be the valencies of the branch points. Let $g'$ be the genus of $\Sigma'$. Then we have:
	\begin{enumerate}
		\item[(i)] (the Hurwitz formula)
		$ 2(g-1)/n = 2(g'-1) + \sum_{i=1}^l \left(1 - 1/ \lambda_i\right).$
		
		\item[(ii)] (Nielsen (\cite{nielsen1937structur}))$ \sum_{i=1}^l\ \sigma_i/\lambda_i$ is an integer.
		
		\item[(iii)] (Wiman (\cite{wiman1896ueber})) $ n \leq 4g + 2 $.
		
		\item[(iv)] (Harvey (\cite{harvey1966cyclic})) Assume $ g \geq 2 $. Set $ M = \mathrm{lcm}(\lambda_1, \ldots, \lambda_l) $. Then we have:
		\begin{enumerate}
			\item[(1)] lcm $(\lambda_1, \ldots, \widehat{\lambda_i}, \ldots, \lambda_l) = M $ for $1 \leq i \leq l $, where $\widehat{\lambda_i}$ denotes the omission of $\lambda_i$.
			
			\item[(2)] $ M|n $, and if $ g' = 0 $, then $ M = n $.
			
			\item[(3)] $ l \neq 1 $, and if $ g' = 0 $, then $ l \geq 3 $.
			
			\item[(4)] If $ 2|M $, the number of $\lambda_1, \ldots, \lambda_l $ which are divisible by the maximal power of $2$ dividing $M $ is even.
		\end{enumerate}
	\end{enumerate}
	Conversely, there exists a periodic map of $\Sigma_g$ corresponding to the data satisfying the above conditions.
\end{prop}

For convenience, we usually denote the valency at each multiple point of $f$ by $\sigma_{1}/\lambda_{1} + \cdots + \sigma_{l}/\lambda_{l}$,
which we call the \emph{total valency}.

Let $\Sigma_g$ be a closed surface of genus $g \geq 2$, and $\mathcal{P}(g)$ (resp. $\mathcal{P}_n(g)$) be the set of all (resp. all $n$-th order) conjugacy classes of periodic maps of $\Sigma_g$.  
By Proposition~\ref{important prop}, we know that
\begin{equation}\label{P(g)}
	\mathcal{P}(g)\cong\{(g',n,\{\lambda_i,\sigma_i\}) ~|~\text {satisfy  (i)-(iv)  of  Proposition}~\ref{important prop}\}.
\end{equation}

\begin{lemma}\label{l}
	With the notation as above, we have
	\begin{enumerate}[(1)]
		\item $g' \leq (g + 1)/2$;
		\item $2 \leq l \leq l_M:=2(g - 2g' + 1)$.
	\end{enumerate}
\end{lemma}

\begin{proof}
	Since $\lambda_i \geq 2$ and $n\geq 2$, the results are directly from the Hurwitz formula, Proposition~\ref{important prop} (i). 
\end{proof}

First we introduce some basics of partition numbers.

\begin{definition}[\cite{Comtet1974}]\label{part}
	A partition of $n$ into exactly $k$ positive integers is an unordered sum of $n$ that uses exactly $k$ positive integers. The number of such partitions is denoted by $p(n, k)$.
\end{definition}
It is known that (\cite{orucc2016number}), for any positive integers $n, k$ with $k\leq n-1$,	
\begin{equation}\label{p(n,k)}
	p(n, k) \leq \frac{5.44}{n - k} {\mathrm e}^{\pi \sqrt{\frac{2(n - k)}{3}}},
\end{equation}
where the right-hand side is an increasing function of  $(n-k)$.

\begin{theorem}\label{Th:P(g)}
	For every integer $n$ with $2 \leq n \leq 4g + 2$, the number $P_n(g) := |\mathcal{P}_n(g)| $ satisfies 
	\[
	\begin{split}
		P_n(g) < \frac{1}{2} \left( \frac{2.72}{3(g + 1)} \right)^2 {\mathrm e}^{\frac{10}{\sqrt{3}} \pi (g + 1)}.
	\end{split}
	\]
	In particular,
	\[
	P(g):=|\mathcal{P}(g)|=\sum_{n=2}^{4g+2}P_n(g)
	<P_M(g):=
	\frac{2}{g+1}\, {\mathrm e}^{\frac{10\pi}{\sqrt{3}}(g+1)} .
	\]
\end{theorem}

\begin{proof}
	We will proceed in several steps.
	
	\emph{Step $1$}.    
	For any $2\leq l\leq l_M$, let $\Lambda_l$ be the set of multi-sets of $\left\{\lambda _ {1}, \lambda _ {2}, \cdots, \lambda _ {l} \right\} $ satifying (i)-(iv) of Proposition~\ref{important prop}. We know that the ramification indices $\lambda_i$ are integers satisfying $2 \leq \lambda_i \leq n$ and $\lambda_i \mid n$. By substituting $m_i = n/\lambda_i$, it follows that the $m_i$ are also positive integers. Then we have 
	\begin{align*}
		|\Lambda_l|&\leq\#\Big\{ \{\lambda_1,\ldots,\lambda_l\}~| ~~\lambda_i|n,~	2(g-1)=2n(g^{\prime}-1)+\sum_{i=1}^{l}(n-n/{\lambda_{i}})\Big\}\\
		&=\#\Big\{\{m_1,\ldots,m_l\}~|~~ m_i|n,~	2(g-1)=2n(g'-1)+\sum_{i=1}^{l}(n-m_i)\Big\}\\
		&\leq\#\{\{m_1,\ldots,m_l\}~| ~~	m_1+\cdots+m_l= 2n(g' - 1) + nl - 2(g - 1)\}\\
		&=p(N_l, l),\,\text{where}\, N_l = 2n(g' - 1) + nl - 2(g - 1)\geq l.
	\end{align*}
	By  Proposition~\ref{important prop} (iii) and Lemma~\ref{l}, we have		
	\begin{align*}
		N_l- l &= 2n(g' - 1) + (n - 1)l - 2(g - 1) \\
		&\leq 2(4g + 2)\left((g + 1)/2 - 1\right) + 2(4g + 1)(g - 2g' + 1) - 2(g - 1) \\
		&\leq (4g + 2)(g - 1) + 2(4g + 1)(g + 1) - 2(g - 1) \\
		&< 18(g + 1)^2.
	\end{align*}
	Hence, for any $2\leq l \leq l_M$, by the inequality~\eqref{p(n,k)},
	\begin{equation}\label{Lambda}
		|\Lambda_l| \leq P(N_l, l) \leq\frac{5.44}{N_l - l} {\mathrm e}^{\pi \sqrt{\frac{2(N_l - l)}{3}}}< \frac{5.44}{18(g + 1)^2} {\mathrm e}^{2\sqrt{3}\pi(g + 1)}.
	\end{equation}
	
	\emph{Step $2$}. Let $\tilde{\lambda}=\left\{ \lambda _ { 1 },  \cdots, \lambda _ { l } \right\}\in \Lambda_l$ and $\Sigma_{\tilde{\lambda}}$ be the set of  multi-sets $\{\frac{\sigma_1}{\lambda_1},\ldots,\frac{\sigma_l}{\lambda_l}\} $ satisfying (ii) of Proposition~\ref{important prop}.    
	Define  the integer
	\[
	t_{\tilde{\lambda}}:= \frac{\sigma_1}{\lambda_1} + \frac{\sigma_2}{\lambda_2} + \cdots + \frac{\sigma_l}{\lambda_l}.
	\]
	Thus,
	\begin{align*}
		|\Sigma_{\tilde{\lambda}}|
		&\leq \# \left\{ \left\{ \frac{\sigma_1}{\lambda_1}, \dots, \frac{\sigma_l}{\lambda_l} \right\} \;\middle|\; \frac{n\sigma_1}{\lambda_1} +  \frac{n\sigma_2}{\lambda_2} \cdots +  \frac{n\sigma_l}{\lambda_l}  = nt_{\tilde{\lambda}} \right\} \\
		&\leq \# \Big\{ \{\tau_1, \dots, \tau_l\} \Bigm| \tau_1 + \cdots + \tau_l = nt_{\tilde{\lambda}} \Big\} \\
		&= p(nt_{\tilde{\lambda}}, l).
	\end{align*}
	By Lemma~\ref{l}, we have
	\[
	\begin{aligned}
		n t_{\tilde{\lambda}} - l &< n l - l \leq 2(4 g + 2 - 1)(g - 2 g' + 1) \leq  8(g + 1)^2.
	\end{aligned}
	\]
	Thus, by the inequality~\eqref{p(n,k)},
	\begin{equation}\label{Sigma}
		|\Sigma_{\tilde{\lambda}}| \leq p(n t_{\tilde{\lambda}}, l)\leq \frac{5.44}{n t_{\tilde{\lambda}} - l} {\mathrm e}^{\pi \sqrt{\frac{2(n t_{\tilde{\lambda}} - l)}{3}}} < \frac{5.44}{8(g + 1)^2} {\mathrm e}^{\frac{4}{\sqrt{3}} \pi (g + 1)}.
	\end{equation}
	
	Note that 	
	\begin{equation*}
		\begin{split}
			P_n(g) = |\mathcal{P}_n(g)|
			= \sum_{g' = 0}^g \sum_{l=2}^{l_M}
			\sum_{\tilde{\lambda} \in \Lambda_l}
			|\Sigma_{\tilde{\lambda}}|.
		\end{split}
	\end{equation*}
	This completes the proof by combining~\eqref{P(g)}, \eqref{Lambda}, and~\eqref{Sigma}.
	\end{proof}
In the above theorem,  we give an upper bound of $P(g)$ in  exponential form. We will show that we cannot expect an upper bound of  $P(g)$ in polynomial form in the next subsection.
	
\subsection{No polynomial upper bound for periodic maps}
First we will give some known inequalities ([\cite{AblowitzFokas2008,Comtet1974}]) which will be used in the following context. For any positive integer $n\in \mathbb{N}^*$, 
\begin{align}\label{Stirling}
	\frac{(n+1)^n}{{\mathrm e}^n} < n! < \frac{(n+1)^{n+1}}{{\mathrm e}^n}.
\end{align}
If the positive integer $k = O(n^{1/3})$, then for $n \to \infty$
\begin{align}\label{comb}
	p(n, k) \sim \frac{1}{k!} \binom{n-1}{k-1}.
\end{align}

\begin{lemma}\label{no_poly-up}
	For $k=\left\lfloor n^{\frac{1}{3}}\right\rfloor$ where \(\lfloor\cdot\rfloor\) denotes the floor function, there is no  upper bound of the function $p(n,k)$ of $n$ in polynomial form.
\end{lemma}

\begin{proof}
	For sufficiently large $n$, by~\eqref{Stirling}, we have
	\begin{align*}
		f(k,n) &:= \frac{1}{k!}\binom{n-1}{k-1} \\
		&> \frac{1}{k^3} \left(\frac{{\mathrm e}^2(n-k+1)}{k^2}\right)^{k-1} \\
		&> \frac{1}{n} \left( \frac{{\mathrm e}^2}{2} n^{1/3} \right)^{n^{1/3}-2}.
	\end{align*}
	For any positive integer $d$,
	\[
	0 \leq \lim_{n \to \infty} \frac{n^d}{f(k,n)} \leq \lim_{n \to \infty} \frac{n^{d+1}}{\Big( \frac{1}{2}{\mathrm e}^2 n^{1/3} \Big)^{n^{1/3}-2}} = 0.
	\]
	Therefore, $f(k,n)$ has no upper bound in polynomial form. By~\eqref{comb}, this completes the proof.
\end{proof}

We now prove that  the number  $P(g)$ of conjugacy classes of  periodic maps of $\Sigma_g$ does not admit a polynomial upper bound.

\begin{theorem}
	There is no polynomial upper bound for $P(g)$.
\end{theorem}

\begin{proof}
	Consider  periodic maps of $\Sigma_g$ with order  $n=q$ which is a prime number.
	By the Hurwitz formula, we have
	\[
	\frac{2(g-1)}{q} = 2(g'-1) + \sum_{i=1}^l \left(1 - \frac1{\lambda_i}\right),
	\]
	Since $ \lambda_i |n$ and $ \lambda_i \geq 2$, we know that $\lambda_i  = q$ for $i=1,\ldots,l$. So for $g'=0$, we have that
	$$l=2+\frac{2g}{q-1}\in \mathbb{N}^*.$$
	When $l=\left\lfloor q^{\frac 13}\right\rfloor$, we obtain that
	\begin{equation*}\label{eq g}
		g=g_q:=(\left\lfloor q^{\frac 13}\right\rfloor-2)\cdot \frac{q-1}{2}
	\end{equation*}
	By (ii) of Proposition~\ref{important prop},  we have
	\[m:=\sum_{i=1}^l\frac{\sigma_i}{\lambda_i}=\frac{\sigma_1+\ldots+\sigma_l}q\in\mathbb{N}^*.\]
	It is possible that $m=1$ by the choice of $l$, so
	\begin{align*}
		P(g_q)\geq&\#\{(g'=0,n=q,\{\lambda_i,\sigma_i\}) ~|~\text {satisfy  (i)-(iv)  of  Proposition}~ \ref{important prop}\}\\
		\geq&\#\Big\{\{ \sigma_1, \ldots, \sigma_l \}|  \sigma_1 + \cdots + \sigma_l = q \Big\}\\
		=&p(q,l).
	\end{align*}    
	Therefore, $P(g)$ admits no polynomial upper bound by Lemma~\ref{no_poly-up}. 
\end{proof}

\section{Dual graph}\label{graph}
\subsection{Pseudo-periodic map}
Let $f: \Sigma_g \to \Sigma_g$ be an orientation-preserving homeomorphism of a closed surface of genus $g$. We call $f$ a \emph{pseudo-periodic map} if it is isotopic to a homeomorphism $f':\Sigma_{g}\to\Sigma_{g}$ satisfying the following conditions:

\begin{enumerate}[(1)]
	\item There exists a disjoint union of simple closed curves $\mathcal{C} = C_1 \cup \cdots \cup C_r$ in $\Sigma_g$ such that $f'(\mathcal{C}) = \mathcal{C}$ (allowing $\mathcal{C}$ to be empty);
	
	\item Let $\mathcal{B} = \Sigma_g \setminus \mathcal{C}$. Then the restriction $f'| _\mathcal{B}: \mathcal{B} \to \mathcal{B}$ is isotopic to a periodic map.
\end{enumerate}

Furthermore, such a collection  $\mathcal{C}$  is called an \emph{admissible system of cut curves} if  every connected component of  $\mathcal{B}$ has negative Euler characteristic. If $\Sigma_g$ has negative Euler characteristic, such a system always exists in the isotopy class of $f$.
Replacing $f$ by its isotopic representative $f'$, we may assume that $f(\mathcal{C}) = \mathcal{C}$. Then for each connected component $C_i$ of $\mathcal{C}$, there exists a minimal integer $\alpha_i\ge1$ such that $f^{\alpha_i}(C_i)=C_i$ and $f^{\alpha_i}$ preserves the orientation of $C_i$. There also exists a minimal integer $L_i\ge1$ such that the restriction $f^{L_i}$ to an annular neighborhood of $C_i$ is a Dehn twist of $e_i$ times.
The rational number 
$s(C_i):=\frac{e_i\,\alpha_i}{L_i}$
is called the \emph{screw number} of $f$ along $C_i$. $f$ is said to be of \emph{negative twist} if $s(C_i)<0$ for all $1\le i\le r$. 
Moreover, $C_i$ is called \emph{amphidrome} if $\alpha_i$ is even and $f^{\alpha_i/2}$ preserves $C_i$ but reverses its orientation; otherwise, $C_i$ is \emph{non-amphidrome}.

\subsection{Weighted graph and dual graph}
Let $\Sigma_{g}$ be a Riemann surface of genus $g$. Consider the decomposition $\Sigma_{g}=\mathcal{B}\cup\mathcal{C}$, where $\mathcal{C}$ is an admissible system of cut curves and $\mathcal{B}=\Sigma_{g}\setminus\mathcal{C}$ is the periodic part. We first construct a weighted graph (see Figure~\ref{fig:weighted_graph}) as follows. Each vertex $v$ of the graph corresponds to a connected component $\mathcal{B}_{v}$ of $\mathcal{B}$. Each edge corresponds to a curve in $\mathcal{C}$ that connects the two distinct associated components of $\mathcal{B}$. Furthermore, we assign to each vertex $v$ the weight pair $(g(v),\rho(v))$, where $g(v)$ is the genus of $\mathcal{B}_{v}$ and $\rho(v)$ is the number of those curves in $\mathcal{C}$ adjacent only to the component $\mathcal{B}_v$ (\cite{ashikaga2002classification}).

\begin{figure}[htbp]
	\centering
	\tikzstyle{gnode} = [draw, circle]
	\begin{tikzpicture}
		\begin{scope}[shift={(-2, 0, 0)}]
			\def\a{2}
			\def\b{1}
			\draw [thick] (0,0) ellipse [x radius=\a, y radius=\b];
			
			\tikzset{
				hole/.pic={
					\draw plot[domain=-1:1, smooth, variable=\x] ({\x}, {0.5*(\x*\x-.8)});
					\draw plot[domain=-1:1, smooth, variable=\x] ({\x}, {0.5*(-\x*\x+.8)});
				}
			}
			
			\pic[scale=0.3] at (-0.66*\a,0) {hole};
			\pic[scale=0.3] at (0,0) {hole};
			\pic[scale=0.3] at (0.66*\a,0) {hole};
			
			\draw [thick] plot[domain=-90:90, smooth, variable=\t] ({0.33*\a+0.1*\a*cos(\t)}, {0.95*\b*sin(\t)});
			\draw [dashed,thick] plot[domain=-90:90, smooth, variable=\t] ({0.33*\a-0.1*\a*cos(\t)}, {0.95*\b*sin(\t)});
			\node at (0.33*\a, 0.95*\b) [above] {$C_3$};
			
			\draw [thick] plot[domain=-90:90, smooth, variable=\t] ({0.05*\a*cos(\t)}, {0.55+0.45*\b*sin(\t)});
			\draw [dashed,thick] plot[domain=-90:90, smooth, variable=\t] ({-0.05*\a*cos(\t)}, {0.55+0.45*\b*sin(\t)});
			\node at (0, \b) [above] {$C_2$};
			
			\draw [thick] plot[domain=-90:90, smooth, variable=\t] ({-0.66*\a+0.05*\a*cos(\t)}, {0.42+0.3*\b*sin(\t)});
			\draw [dashed,thick] plot[domain=-90:90, smooth, variable=\t] ({-0.66*\a-0.05*\a*cos(\t)}, {0.42+0.3*\b*sin(\t)});
			\node at (-0.66*\a, 0.8*\b) [above] {$C_1$};
		\end{scope}
		
		\path [draw, latex-latex] (.25,0) -- (1.25, 0);
		
		\begin{scope}[shift={(2, 0, 0)}]
			\node [gnode] at (0,0) (v1) {2};
			\path (v1)+(2,0) node [gnode] (v2) {1};
			\node at (v1.south) [below] (v1lbl) {$v_1$};
			\node at (v2.south) [below] (v2lbl) {$v_2$};
			\node at (v1lbl.south) [below] {$\rho(v_1)=2$};
			\node at (v2lbl.south) [below] {$\rho(v_2)=0$};
			
			\path [draw] (v1) -- (v2);
		\end{scope}
	\end{tikzpicture}
	\caption{Weighted graph}
	\label{fig:weighted_graph}
\end{figure}
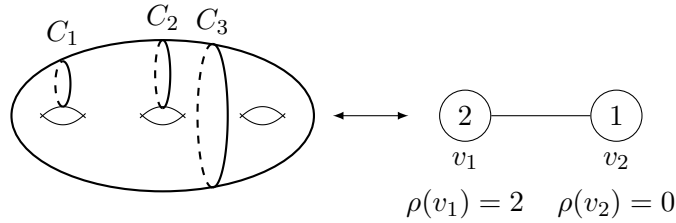

Before constructing the dual graph, we first introduce the stable curve.
Let $F$ be a connected projective curve of genus $g \ge 2$ with irreducible components $C_1, \dots, C_l$. We say that $F$ is \emph{stable} (\cite{DeligneMumford1969}) if
\begin{enumerate}[(1)]
	\item $F$ is reduced;
	\item $F$ has only ordinary double points;
	\item if $C_i$ is a smooth rational component, then $C_i$ meets the other components of $F$ in at least 3 points.
\end{enumerate}

To any stable curve $F$, we associate a dual graph $G_F$ constructed as follows. Each vertex of $G_F$ corresponds to an irreducible component of $F$, and each edge of $G_F$ corresponds to a singular point of $F$. Specifically, for a vertex $v$ associated with an irreducible component $\Gamma$, the self-loops at $v$ correspond to the nodes of $F$ lying entirely on $\Gamma$. There are exactly $p_a(\Gamma) - g(\widetilde{\Gamma})$ such loops, where $\widetilde{\Gamma}$ denotes the normalization of $\Gamma$ \cite{xiao1992fibrations}.  An example of a dual graph is shown in Figure~\ref{fig:dual_graph}.

\begin{figure}[htbp]
	\centering
	
	\begin{tikzpicture}

		\begin{scope}[shift={(0, 0.4)}]
			\tikzset{
				c1curve/.pic={
					\draw[thick] (0,0) to[out=-10, in=320] (3,2);
					\draw[thick,bend right] (3,2) to[out=-120, in=250] (3.3,1.2);
					\draw[thick] (3.3,1.2) to[out=-10, in=340] (4.4,2.8);
					\draw[thick,bend right] (4.4,2.8) to[out=-120, in=250] (4.7,2.1);
					\draw[thick] (4.7,2.1) to[out=-10, in=220] (6,3);
				},
				c2curve/.pic={
					\draw[thick,rounded corners,bend left] (0.5,1) to (5,-2);
				}
			}
			\pic [scale=0.6] at (0,0) {c1curve};
			\pic [scale=0.6] at (-0.3, 0) {c2curve};
			
			\node at (3.3, 1.0) {$p_a=2$};
			\node at (3.8, 1.9) {$C_1$};
			\node at (2.8, -1.5) {$C_2$};  
			\node at (1.8, -0.8) {$g=1$};  
		\end{scope}

		\path [draw, latex-latex] (4.5,0.5) -- (5.8, 0.5);
		
		\begin{scope}[shift={(0, -0.5)}]
			
			\draw[thick] (8, 1) to (9.65, 1);
			\draw[thick] (8,1) to[out=100, in=180, min distance=1.6cm] (8,1);
			\draw[thick] (8,1) to[out=260, in=180, min distance=1.6cm] (8,1);
			\draw[fill=white, thick] (8, 1) circle (5pt);
			\draw[fill=black] (9.65, 1) circle (5pt);
			
			\node at (8.1, 0.55) {$v_1$};
			\node at (9.7, 0.55) {$v_2$};
			\node at (10.5, 1.0) {$g=1$};
		\end{scope}
		
	\end{tikzpicture}
	
	\caption{Dual graph}
	\label{fig:dual_graph}
\end{figure}
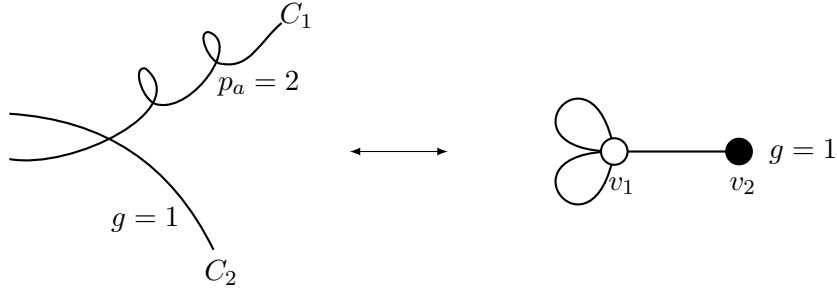

It is evident that there exists a canonical one-to-one correspondence among the following four objects: stable curves, dual graphs, Riemann surfaces with cut curves, and weighted graphs. In particular, for a stable curve $F$ of genus $g$, we denote its associated weighted graph by $G_F$, which is called a \emph{stable weighted graph} of genus $g$.

Let $G_F$ be a stable weighted graph of genus $g$. 
For convenience, we define the numerical weight of a vertex $v$ by $w(v) := g(v) + \rho(v)$, and the total numerical weight of $G_F$ by $w(G_F) := \sum_{v \in V(G_F)} w(v)$. Then
\begin{equation}\label{eq:genus-graph}
	g = w(G_F) + b_1(G_F) = \sum_{v \in V(G_F)} w(v) + b_1(G_F),
\end{equation}
where $b_1(G_F)$ is the first Betti number of $G_F$ (in other words, the number of holes in $G_F$).

\subsection{The number of stable weighted graphs}
Let $\mathcal{G}_g$ denote the set of all stable weighted graphs of genus $g$. 
For any $G \in \mathcal{G}_g$, let $V(G)$ and $E(G)$ be the sets of vertices and edges of $G$, respectively. Then the following results are well known in graph theory.

\begin{lemma}[\cite{West1996}]\label{lem:graph}
	If $G$ is a finite connected undirected graph, then
:	\begin{enumerate}[(1)]
		\item $\sum_{v \in V(G)} \deg(v) = 2|E(G)|.$
		\item  $b_1(G)  = |E(G)| - |V(G)| + 1.$
	\end{enumerate}
\end{lemma}
As an immediate corollary, the genus of a stable weighted graph $G$ satisfies
\begin{equation}\label{eq:genus}
	g(G) = w(G) + |E(G)| - |V(G)| + 1.
\end{equation}

To derive an upper bound on $|\mathcal{G}_g|$, we first establish upper bounds on $|V(G)|$ and $|E(G)|$ for any $G \in \mathcal{G}_g$.

\begin{lemma}\label{Th:VandE}
	If $G$ is a stable weighted graph of genus $g\geq 2$, then
	\begin{equation*}
		|V(G)|\leq 2g-2, \qquad |E(G)|\leq 3g-3.
	\end{equation*}
\end{lemma}

\begin{proof}
	We may assume that $|V(G)| \geq 2$. 
	Let $n_0$ and $n_+$ denote the numbers of vertices of $G$ with $w(v) = 0$ and $w(v) \geq 1$, respectively. Since $G$ is stable, every vertex $v\in V(G)$ with $w(v) = 0$ satisfies $\deg(v) \geq 3$. By Lemma~\ref{lem:graph}, we have
	\begin{equation*}
		2|E(G)| = \sum_{v \in V(G)} \deg(v) \geq 3n_0 + n_+ = 3|V(G)| - 2n_+.
	\end{equation*}
	Substituting this into~\eqref{eq:genus}, we obtain
	\begin{equation*}
		g \geq \frac{|V(G)| - 2n_+}{2} + w(G) + 1.
	\end{equation*}
	So we have that
	\begin{equation*}
		|V(G)|\leq 2g-2+2n_+-2w(G)\leq 2g-2.
	\end{equation*}
	Furthermore,
	\begin{equation*}
		|E(G) |= b_1(G) + |V(G) |- 1 = g - w(G) + |V(G)| - 1\leq 3g-3.
	\end{equation*}   
\end{proof}

We now derive an upper bound on $|\mathcal{G}_g|$.
\begin{theorem} \label{Th:Gg}
	For any integer $g\geq 2$, the number of stable weighted graphs of genus $g$ satisfies
	\[	|\mathcal{G}_g|
	< \frac{81}{4{\mathrm e}^3(g+1)} \left( \frac{2{\mathrm e}}{3}(g+1) \right)^{\!3g}.\]
\end{theorem}

\begin{proof}
	We denote by $\mathcal{G}(n,m)$ the set of connected graphs with $n$ vertices and $m$ edges, and by $\mathcal{G}_g(n,m)$ the set of  stable weighted graphs of genus $g$ with $n$ vertices and $m$ edges.

     For any $G\in \mathcal{G}_{g}(n,m)$, let $V(G)=\{v_1,\ldots,v_n\}$ and $w(G)=\sum_{i=1}^n w_i= \sum_{i=1}^n(g_i+\rho_i)$ where $w_i=w(v_i)$, $g_i=g(v_i)$, and $\rho_i=\rho(v_i)$. For each vertex $v_i$ with $\rho_i \neq 0$, we attach $\rho_i$ loops $\gamma_{i1}, \dots, \gamma_{i\rho_i}$ at $v_i$. Let $G'$ denote the resulting graph. By construction, we have 
     \[
     |V(G')| = |V(G)|, \quad |E(G')| = |E(G)|+\sum_{i=1}^n \rho_i, \quad \text{and} \quad b_1(G') = b_1(G) + \sum_{i=1}^n \rho_i.
     \]
     Thus, $G' \in \mathcal{G}(n, m')$, where $m' := |E(G')|$. Furthermore, by~\eqref{eq:genus}, $m'$ satisfies
     \[
     m' = m + \sum_{i=1}^n \rho_i = g + n - 1 - \sum_{i=1}^n g_i \leq g + n - 1.
     \]
     
     For any graph $G' \in \mathcal{G}(n,m')$, let $V(G')^{(2)}$ denote the set of unordered pairs $[v_i, v_j]$ of vertices of $G'$, where $v_i$ and $v_j$ are not necessarily distinct. Therefore, 
     \[T(n):=\left|V(G')^{(2)}\right|=\frac{n(n+1)}{2}.\]
     Since $G'$ may contain multiple edges. By the construction of $G'$ from $G$, the map $G \mapsto G'$ is injective, so
     \[
     |\mathcal{G}_g(n,m)| \leq |\mathcal{G}(n,m')| \leq \binom{T(n)+m'-1}{m'}.
     \]
     Furthermore,
     \[	|\mathcal{G}_g|
     \leq \sum_{n=1}^{2g-2}\sum_{m=0}^{3g-3} |\mathcal{G}_g(n,m)| \leq (3g-2) \sum_{n=1}^{2g-2} \binom{T(n)+m'-1}{m'}.\]
     
     Since $T(n) = \frac{n(n+1)}{2}$ is increasing in $n$, and $m' \leq g+n-1$ is also increasing in $n$, while $\binom{T+m'-1}{m'}$ is increasing in both parameters, it follows that $\binom{T(n)+m'-1}{m'}$ is increasing in $n$ for $n \geq 1$. Therefore,
     \[ \sum_{n=1}^{2g-2} \binom{T(n)+m'-1}{m'} \le (2g-2) \cdot \binom{T_{\max}+m'_{\max}-1}{m'_{\max}}, \]
     where
     \[ T_{\max} = 2g^2-3g+1,\quad \text{and} \quad m'_{\max} = 3g-3. \]
     Thus, by~\eqref{Stirling}, we obtain
     \begin{align*}
     	|\mathcal{G}_g|& \le (3g-2)(2g-2) \binom{2g^2-3}{3g-3}\\
     	&\leq (3g-2)(2g-2) \left( \frac{2{\mathrm e}}{3}(g+1) \right)^{\!3g-3} \\
     	&< \frac{81}{4{\mathrm e}^3(g+1)} \left( \frac{2{\mathrm e}}{3}(g+1) \right)^{\!3g}.
     \end{align*}
\end{proof}

\section{Automorphisms of weighted graphs}

Let $G$ be a stable weighted graph of genus $g$. An \emph{automorphism} $\sigma$ of $G$ consists of two bijections $\sigma_V:V(G)\to V(G)$ and $\sigma_E:E(G)\to E(G)$ such that
\begin{enumerate}[(1)]
	\item for every edge $e\in E(G)$ with endpoints $v_1$ and $v_2$, $\sigma_E(e)$ has endpoints $\sigma_V(v_1)$ and $\sigma_V(v_2)$;
	
	\item for every vertex $v \in V(G)$, the weight pair is preserved, i.e., 
	$(g(v), \rho(v)) = \bigl( g(\sigma_V(v)), \rho(\sigma_V(v)) \bigr).$
\end{enumerate}
     
All graphs considered here are undirected. For notational convenience in describing the automorphisms, we fix in advance a reference orientation $\vec{e}$ for each edge $e \in E(G)$ of the weighted graph $G$. We give an example.

\begin{example}\label{E11}
	The graph of type $E_{11}$(which means $(g(v_i),\rho(v_i)=(1,0)),i=1,2$.) has the following four  automorphisms (Figure~\ref{fig11}):
	
	\begin{figure}[htbp]
		\centering
		\begin{tikzpicture}
			\usetikzlibrary{positioning}
			\node [shape=circle,draw] (v1) at (0,0) {1};
			\node [shape=circle,draw] (v2) at (1.5,0) {1};
			\node [below=8pt] at (v2)  {$v_2$};
			\node [below=8pt] at (v1)  {$v_1$};
			\draw[->]  (v2)to [bend left=20] node [below]{\tiny{${e_1}$}} (v1);
			\draw[->] (v2)to [bend left=-20] node [above]{\tiny{${e_2}$}} (v1);
		\end{tikzpicture}
		\caption{Weighted graph $E_{11}$}
		\label{fig11}
	\end{figure}
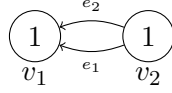
	
	\begin{enumerate}[(1)]
		\item $\sigma$ acts trivially.
		
		\item $\sigma$ fixes $v_1$ and $v_2$, and exchanges the edges $e_1$ and $e_2$.
		
		\item $\sigma$ exchanges $v_1$ and $v_2$, and fixes the set $\{e_1,e_2\}$, but reverses their orientations.
		
		\item $\sigma$ exchanges $v_1$ and $v_2$, and exchanges $e_1$ to $-e_2$, and $e_2$ to $-e_1$.
	\end{enumerate}
\end{example}

Let $\mathrm{Aut}(G)$ be the group of all automorphisms of $G$.

\begin{theorem}\label{Th:AutG}
	Let $G$ be a stable weighted graph of genus $g\ge 2$. Then
	\[
	|\mathrm{Aut}(G)| \le 2^{3g-3}\,(3g-3)! <\left( \frac{2}{e} \right)^{3g-3} (3g-2)^{3g-2}.
	\]
\end{theorem}

\begin{proof}
	Let $G$ be a stable weighted graph of genus $g\ge 2$ with $m$ edges and let
	\[
	E(G)=\{e_1,\dots,e_m\},\quad \text{and} \quad
	\overrightarrow{E}(G)=\{\pm\vec e_i\mid 1\le i\le m\}.
	\]
	For each $1 \le i \le m$, let $u_i, v_i \in V(G)$ be the endpoints of $e_i$ such that $\vec{e}_i$ is directed from $u_i$ to $v_i$, and we write $\vec{e}_i=(u_i,v_i)$. Each automorphism $\sigma=(\sigma_V,\sigma_E)\in \mathrm{Aut}(G)$ induces a permutation $\pi_\sigma$ in the symmetric group $S_m$ on the $m$ edges, defined by $\sigma_E(e_i)=e_{\pi_\sigma(i)}$.
	Define
	\[X(G):=\left\{\sigma_{\overrightarrow{E}}:\overrightarrow{E}(G) \to \overrightarrow{E}(G)\ \text{is a bijection}\;\middle |\;\sigma_{\overrightarrow{E}}(- \vec e_i)=- \sigma_{\overrightarrow{E}}(\vec e_i),\forall \vec e_i\in\overrightarrow E(G)\right\}
	\]
	and
	\[
	\Phi: \mathrm{Aut}(G) \longrightarrow X(G), \qquad
	\sigma \longmapsto \Phi(\sigma),
	\]
	where the map $\Phi(\sigma): \overrightarrow{E}(G) \to \overrightarrow{E}(G)$ satisfies
	\[
	\Phi(\sigma)(\vec{e}_i) =
	\begin{cases}
		\phantom{-}\vec{e}_{\pi_\sigma(i)}, & \text{if } 
		(\sigma_V(u_i),\sigma_V(v_i))
		=(u_{\pi_\sigma(i)},v_{\pi_\sigma(i)}),\\
		-\vec e_{\pi_\sigma(i)}, & \text{if }
		(\sigma_V(u_i),\sigma_V(v_i))
		=(v_{\pi_\sigma(i)},u_{\pi_\sigma(i)}).
	\end{cases}
	\]
	It is easy to see that $\Phi$ is injective.
	Consequently,
	\[
	|\mathrm{Aut}(G)| \le |X(G)| 
	= \bigl| S_m \times \{\pm 1\}^m \bigr| 
	= 2^m \, m!.
	\]
	By Theorem~\ref{Th:VandE}, 
	$m=|E(G)|\le 3g-3$, so
	\[
	|\mathrm{Aut}(G)|\le 2^{3g-3}\,(3g-3)!.
	\]
\end{proof}

\section{Marked generalized quotient space}
\subsection{Generalized quotient}
Let $f:\Sigma_g\to\Sigma_g$ be a pseudo–periodic map of negative twist. Matsumoto-Montesinos (\cite{matsumoto2011pseudo}) constructed the
minimal generalized quotient
$\pi:\Sigma_g\to S_f,$ where $S_f$ is a numerical chorizo space, i.e., $S_f$ is the underlying topological space of a Riemann surface with nodes
so that a multiplicity $m_j$ is attached to each component $S_f^{j}$ of $S_f$. 

Let $\Sigma_g = \mathcal{A }\cup \tilde{\mathcal{B}}$, where $\mathcal{A}$ is the union of annular neighborhood of the curve in $\mathcal{C}$, and $\tilde{\mathcal{B}}$ is the closure of $\Sigma_g-\mathcal{A}$.  For each  $C_i \in \mathcal{C}$,  let
$\mathcal A_i$ be an annular neighborhood of $C_i$. The integer
$K$ associated with $C_i$ is given by
\[
K=
\begin{cases}
	\displaystyle
	-s(C_i)
	-\frac{\delta^{(1)}}
	{\lambda^{(1)}}
	-\frac{\delta^{(2)}}
	{\lambda^{(2)}},
	& \text{if $C_i$ is non-amphidrome},\\[3mm]
	\displaystyle
	-\frac{s(C_i)}{2}
	-\frac{\delta}
	{\lambda},
	& \text{if $C_i$ is amphidrome}.
\end{cases}
\]
In the non-amphidrome case,
$(m^{(j)},\lambda^{(j)},\sigma^{(j)})$, $j=1,2$, are the
valencies of the two boundary components of $\mathcal A_i$, and
$\sigma^{(j)}\delta^{(j)}
\equiv1\pmod{\lambda^{(j)}},
0\leq\delta^{(j)}<\lambda^{(j)}.$
In the amphidrome case, the two boundary components have the same
valency
$(2m,\lambda,\sigma),$
and $\delta$ is determined by
$\sigma\delta\equiv1\pmod{\lambda},
0\leq\delta<\lambda.$
Then $K$ is an integer with
$K\geq -1$.

\subsection{Classification of marked generalized quotients}\label{sec:quo}
Let $\Sigma$ be a surface of genus $g_0$ with $k$ boundary components $\partial_1, \ldots, \partial_k$. Let $f : \Sigma \to \Sigma$ be an orientation-preserving homeomorphism satisfying the following conditions: 
(1) there is a collection of pairwise disjoint simple closed curves $\mathcal{C} = \bigsqcup_{j=1}^r C_j$ contained in the interior of $\Sigma$;
(2) $\Sigma \setminus \mathcal{C}$ is connected; 
(3) $f(\mathcal{C}) = \mathcal{C}$ and the restriction $f|_{\Sigma \setminus \mathcal{C}}$ is periodic.
Let $h$ denote the genus of the periodic part $\Sigma \setminus \mathcal C$, so that $h=g_0-r$.

Ashikaga and Ishizaka introduced the marked generalized quotient $M_f$ by adding the boundary marking data induced by $f$ (for details on the classification, we refer to \cite[Section~2]{ashikaga2002classification}). This quotient may still depend on the integer parameter $K$. However, in their tables, the resulting graphs are encoded by unified graph types, where different values of $K$ are represented by a single type. In this section, we establish upper bounds for the number of such unified marked generalized quotient types in different genera.

Let $\mathcal N(h;k,r)$ be the set of all unified marked generalized quotient types in the sense of Ashikaga and Ishizaka, associated with the periodic part of genus $h$ with $k$ original boundary components and $r$ cut curves.

\begin{theorem}\label{gi=0}
	For $h=0$, the number $N(0;k,r):=|\mathcal N(0;k,r)|$ satisfies
	\[
	N(0;k,r) \leq
	1+4(r+2)+4\sum_{n=3}^{\max\{k,2r\}}\varphi(n)
	\left(\left\lfloor \frac{r}{n}\right\rfloor+1\right),
	\]
	where $\varphi$ denotes Euler's totient function.
\end{theorem}

\begin{proof}
	Let $n$ be the order of the periodic map. If $n=1$, the periodic map is necessarily the identity and hence gives rise to only one type. Now consider the case $n \ge 2$. Because the periodic part has genus 0, capping off its boundary components with disks extends the map to an orientation-preserving periodic homeomorphism of order $n$ on the sphere. Such a map has exactly two fixed points, and every other point has period $n$ (\cite{conejeros2018periodic}). Consequently, for each original boundary component of the surface and each boundary curve arising from cutting along $\mathcal C$, the first entry $m$ in its valency data $(m,\lambda,\sigma)$ can only be $1$ or $n$.
	
	 First, we classify the original boundary components and the boundary curves arising from cutting along $\mathcal{C}$ according to whether they are fixed by the periodic map $f$ or belong to an $n$-cycle under $f$.
	
	For the $k$ original boundary components, let $k_1$ be the number of these components fixed by $f$, and let $k_2$ be the number of $n$-cycles formed by the remaining original components under the action of $f$. It follows that
	\[ k = k_1 + n k_2, \quad k_1, k_2 \in \mathbb{Z}_{\ge 0}. \]
	
	 On the other hand, cutting along the $r$ cut curves gives rise to $2r$ boundary curves. We classify these $2r$ boundary curves arising from the cut curves into three categories:
	\begin{enumerate}[(1)]
		\item $r_1$: the number of non-amphidrome cut curves fixed by $f$. Each such curve gives rise to two boundary curves, both of which are fixed by $f$. Hence, these cut curves contribute $2r_1$ boundary curves.
		\item $r_2$: the number of $n$-cycles of non-amphidrome cut curves under $f$. Each such cycle consists of $n$ cut curves. Since these cut curves are non-amphidrome, the two sides of the cut curves are not interchanged by $f$. Therefore, the $2n$ boundary curves arising from these $n$ cut curves are divided into two $n$-cycles.  Hence, these cut curves contribute a total of $2nr_2$ boundary curves.
		\item $r_3$: the number of $n$-cycles of boundary curves arising from amphidrome cut curves under $f$, contributing $nr_3$ boundary curves in total.
	\end{enumerate}
	It follows that
	\[2r = 2r_1 + 2n r_2 + n r_3, \quad r_1, r_2, r_3 \in \mathbb{Z}_{\ge 0}.\]
	
	 An orientation-preserving periodic map of order $n\ge 2$ on $S^2$ has exactly two fixed points (\cite{conejeros2018periodic}). It follows that
	\[
	k_1+2r_1\leq 2.
	\]
	Thus, there are at most four choices for the pair $(k_1, r_1)$: $(0,0)$, $(1,0)$, $(2,0)$, and $(0,1)$. For a fixed $n$, the integer $r_2$ can take the values $0, 1, \dots, \left\lfloor \frac{r}{n} \right\rfloor$, giving at most $\left\lfloor \frac{r}{n} \right\rfloor + 1$  choices. Once $k_1, r_1$ and $r_2$ are fixed, $k_2$ and $r_3$ are uniquely determined.  Consequently,  for a fixed $n$, we have
	\[
	\# \left\{ (k_1, k_2, r_1, r_2, r_3) \in \mathbb{Z}_{\ge 0}^5 \;\middle|\;
	\begin{aligned}
		k &= k_1 + nk_2, \\
		2r &= 2r_1 + 2nr_2 + nr_3, \\
		k_1 &+ 2r_1 \le 2
	\end{aligned}
	\right\} \le 4 \left( \left\lfloor \frac{r}{n} \right\rfloor + 1 \right).
	\]
	
	Secondly, we consider the valency data on the sphere. 
	By the Hurwitz formula, the cyclic cover induced by a periodic map of order $n$ on $S^2$ has exactly two branch points, and the corresponding valency data is given by the unordered pair $\left\{ \frac{\sigma}{n}, \frac{n-\sigma}{n} \right\},$ where $\sigma$ is an integer satisfying $1 \le \sigma \le n-1$ and $\gcd(\sigma, n) = 1$.
	Therefore, for a fixed $n$, we have
	\[
	\# \left\{ \left\{ \frac{\sigma}{n}, \frac{n - \sigma}{n} \right\} \;\middle|\; 1 \le \sigma \le n - 1, \gcd(\sigma, n) = 1 \right\} = 
	\begin{cases} 
		1, & n = 2, \\ 
		\dfrac{\varphi(n)}{2}, & n \ge 3,
	\end{cases}
	\]
	where $\varphi$ denotes Euler's totient function. 
	Thus, for a fixed order $n$, the number $N_0(n;k,r)$ of possible valency combinations for periodic maps of order $n$ on a periodic part of genus $0$ with $k$ original boundary components and $r$ cut curves satisfies
	\[
	N_0(n;k,r) \leq
	\begin{cases}
		2(r+2), & n=2, \\
		2\varphi(n)\left(\left\lfloor\frac{r}{n}\right\rfloor+1\right), & n\geq 3.
	\end{cases}
	\]
	Summing over all possible orders $n$, we conclude that the total number $N_0(k, r)$ of such combinations satisfies
	\[
	N_0(k,r) \leq 1 + 2(r+2) + 2\sum_{n=3}^{\max\{k, 2r\}} \varphi(n) \left( \left\lfloor \frac{r}{n} \right\rfloor + 1 \right).
	\]
	
	 Finally, assigning the unordered pair 
	$\left\{\frac{\sigma}{n},\frac{n-\sigma}{n}\right\}$
	to the two branch points gives at most two choices. Hence, we obtain
	\[
	N(0;k,r) \leq
	1+4(r+2)+4\sum_{n=3}^{\max\{k,2r\}}\varphi(n)
	\left(\left\lfloor\frac{r}{n}\right\rfloor+1\right).
	\]
\end{proof}

\begin{theorem}\label{gi>0}
	For $h>0$, the number $N(h;k,r):=|\mathcal N(h;k,r)|$ satisfies
	\[
	N(h;k,r) \leq 2^r \binom{2h+3}{h+1}^{2}P(h) < 2^{r+4h+6}P(h) ,
	\]
	where $P(h)$ denotes the number of conjugacy classes of periodic maps on the closed surface of genus $h$.
\end{theorem}

\begin{proof}
	A marked generalized quotient space is essentially valency  with markings for the original boundary components and the boundary curves arising from the cut curves.
	By Lemma~\ref{l}, the number $l$ of branch points of the quotient map induced by a periodic map on a surface of genus $h$ satisfies $l\leq 2h+2$. Together with the trivial valency $1$, there are at most $l+1\leq 2h+3$ possible valencies to choose from. First, for the $k$ original boundary components, the number of possible valency choices is bounded above by
	\[ \binom{l+1}{k}\leq \binom{2h+3}{k}\leq \binom{2h+3}{h+1}. \]
	Similarly, for the boundary curves arising from the cut curves, the number of possible valency choices is also at most $\binom{2h+3}{h+1}.$  Moreover, we also need to consider whether the cut curves are amphidrome. In summary, we obtain an upper bound for $N(h;k,r)$, i.e.
	\[N(h;k,r) \leq 2^r \binom{2h+3}{h+1}^{2}P(h).\]
	where $P(h)$ denotes the number of conjugacy classes of periodic maps on the closed surface of genus $h$.
\end{proof}

\subsection{Substitution of marked generalized quotients}
Let $G$  be a stable weighted graph of genus $g$ and $\sigma$ be an automorphism of $G$. Following Ashikaga and Ishizaka, we consider the quotient graph $ \widetilde{H} = G /\sigma. $  The substitution procedure replaces the local part of $\widetilde{H}$ around each vertex by a suitable marked generalized quotient. For details, we refer the reader to Ashikaga and Ishizaka (\cite[Section~3]{ashikaga2002classification}).

By the Matsumoto--Montesinos theorem (\cite{matsumoto1994pseudo}),
the classification of  pseudo-periodic maps of negative twist is equivalent to the classification of triples
$(G, \sigma, S_f)$  in our notation.
For fixed $G$ and $\sigma$, the space $S_f$ is constructed by
assigning to each vertex $v_i$ of the quotient graph
$\widetilde H=G/\sigma$
a compatible marked generalized quotient $M_{v_i}$ determined by the
corresponding values $(g_i,r_i,k_i)$.

\begin{corollary}\label{N0}
	For $g_i = 0$, let $N(0):=\max_{k_i,r_i} N(0;k_i,r_i)$. Then 
	\[ N(0) \leq 16g^2. \]
\end{corollary}

\begin{proof}
	By Theorem~\ref{gi=0}, this follows immediately from
	$r_i\leq g,k_i\leq 2g$ and $\varphi(n)\leq n.$ 
\end{proof}

Fix a pair $(G,\sigma)$, where $G$ is a stable weighted graph of genus $g$ and $\sigma$ is an automorphism of $G$. For the quotient graph $\widetilde{H}=G/\sigma$, let $S(G,\sigma)$ be the set of all possible assignments of unified marked generalized quotient types to the vertices of $\widetilde{H}$.

\begin{theorem}\label{Th:Sf}
	For any stable weighted graph $G$ of genus $g \ge 2$ and any automorphism $\sigma$ of $G$, we have
	\[
	|S(G,\sigma)|
	< 2^{16g-12} e^{\frac{10\pi}{\sqrt{3}}(3g-2)} g^{4g-4}.
	\]
\end{theorem}

\begin{proof}
	 Suppose that the quotient graph $\widetilde H$ has vertices $v_1,\ldots,v_V$. For each $i=1,\ldots,V$, let $(g_i,k_i,r_i)$ be the data associated with $v_i$. Let $V_0$ and $V_+$ denote the numbers of vertices of genus $0$ and of positive genus, respectively.
	
	By Corollary~\ref{N0},  for a genus $0$ vertex, $N(0)\leq 16g^2$; by Theorem~\ref{gi>0}, for a vertex of positive genus,  $ N(g_i;k_i,r_i) \leq 2^{r_i} \binom{2g_i+3}{g_i+1}^{2} P(g_i). $ It follows that
	\begin{align*}
		|S(G,\sigma)|& \leq (16g^2)^{V_0} \prod_{\substack{1\leq i\leq V\\ g_i>0}} \left[ 2^{r_i} \binom{2g_i+3}{g_i+1}^{2} P(g_i) \right]\\
		&<(16g^2)^{V_0}  \prod_{\substack{1\leq i\leq V\\ g_i>0}} \left[2^{r_i+4g_i+6}P(g_i) \right],
	\end{align*}
	where $P(g_i) < \frac{2}{g_i+1}\, \mathrm{e}^{\frac{10\pi}{\sqrt{3}}(g_i+1)}$ (by Theorem~\ref{Th:P(g)}).
	
	 Recall from equation~\eqref{eq:genus-graph} that $g = \sum_i (g_i + r_i) + b_1(G)$, which implies that
	\[ \sum_{g_i>0} (g_i + r_i) \leq g \quad \text{and} \quad \sum_{g_i>0} g_i \leq g. \]
	Therefore, we obtain
	\[ |S(G,\sigma)| < (16g^2)^{V_0} \, 2^{4g+6V_+} \, e^{\frac{10\pi}{\sqrt{3}}(g+V_+)}. \]
	Furthermore, by Lemma~\ref{Th:VandE}, we have $V_0+V_+ = V \leq 2g-2$. After simplification, we finally obtain
	\begin{align*}
		|S(G,\sigma)| < 2^{16g-12} e^{\frac{10\pi}{\sqrt{3}}(3g-2)} g^{4g-4}.
	\end{align*}
\end{proof}
	
\section{An upper bound on the number of singular fibers}
Let $\phi: S \to \Delta$ be a proper surjective holomorphic map from a complex surface $S$ to the unit disk $\Delta = \{ t \in \mathbb{C} \mid |t| < 1 \}$, such that the fiber $\phi^{-1}(t)$ is a smooth Riemann surface of genus $g \geq 2$ for any $t \in \Delta^* = \Delta \setminus \{0\}$. The map $\phi$ is called a degeneration of genus $g$, and $ F = \phi^{-1}(0) $ is called the \emph{singular fiber} (\cite{ashikaga2002classification}). 
Suppose that $\phi$ is \emph{normally minimal}, i.e., 
the reduced scheme of $F$ has only normal crossing and every $(-1)$-curve in $F$ intersects the other components at least three points.

Let $\phi_1: S_1 \to \Delta$ and $\phi_2: S_2 \to \Delta$ be degenerations of genus $g$.
We say that $\phi_1$ and $\phi_2$ are \emph{topologically equivalent} if
there is an orientation-preserving homeomorphism
$\psi : S_1 \to S_2$ such that  $\phi_2 \circ \psi = \phi_1$. 
We denote by
\[
\mathcal D_g
=
\{[\phi] \mid \phi \text{ is a degeneration of genus } g\},
\]
where $[\phi]$ denotes the topological equivalence class of $\phi$. 

Next, we will use Matsumoto-Montesinos' theorems to classify topological equivalence classes of  degenerations.

\begin{theorem}\upshape(\cite{matsumoto2011pseudo})\label{MM1}
	The conjugacy class, realized as the topological monodromy of a degeneration of curves, of the mapping class group of a Riemann surface of genus $g \geq 2 $ is represented by a pseudo-periodic map of negative twist. Conversely, any conjugacy class of pseudo-periodic map of negative twist is realized as the topological monodromy of certain degeneration of curves.
\end{theorem}

\begin{theorem}\upshape(\cite{matsumoto2011pseudo})\label{MM2}
	Let $\Sigma_g$ be a closed oriented surface of genus $g\ge1$, and let
	$f:\Sigma_g\to\Sigma_g$ be a pseudo-periodic map of negative twist.
	Then the conjugacy class of $f$ is completely determined by the following data:An admissible system of cut curves $\mathcal C=\bigsqcup C_i$ on $\Sigma_g$,the action of $f$ on the oriented graph $G_{\mathcal C}$ induced by $\mathcal C$, the screw numbers of $f$ around the annuli of each $C_i$ and the valency data of the periodic maps which stabilize the connnected components of $\Sigma_g-\mathcal{C}$. It is equivalent to determining the action of $f$ on $G_{\mathcal{C}}$ and Matsumoto-Montesinos' generalized quotient $\Sigma_g \to S_f$ of $f$.	
\end{theorem}

By Theorem~\ref{MM1}, the topological equivalence classes of degenerations of genus $g$ are in one-to-one correspondence with the conjugacy classes of pseudo-periodic maps of negative twist on $\Sigma_g$. Hence, the classification of degenerations is equivalent to the classification of pseudo-periodic maps of negative twist.

By Theorem~\ref{MM2}, the conjugacy classes of
pseudo-periodic maps of negative twist are classified by triples
$(G,\sigma,S_f),$
where $G$ is a stable weighted graph of genus $g$, $\sigma$ is an
automorphism of $G$, and $S_f$ is the minimal generalized quotient of $f$.
For fixed $G$ and $\sigma$, the minimal generalized quotient $S_f$ is
obtained by substituting compatible marked generalized quotients into
the vertices of the quotient graph $G/\sigma$.
As discussed in the second paragraph of Section~\ref{sec:quo}, we classify the marked generalized quotient spaces, which may vary with the integer parameter $K$ determined by the screw number and the boundary valencies, into finitely many unified marked generalized quotient types. Consequently, we obtain an upper bound $|\mathcal{G}_{g}| \cdot |\text{Aut}(G)| \cdot |S(G,\sigma)|$ on the number of classes of pseudo-periodic maps of negative twist of genus $g$. By Theorem~\ref{MM1}, this is equivalent to the classification of degenerations of curves of genus $g$. In this way, we obtain finitely many unified types of degenerations.

\begin{theorem}
    For any $g\geq 2$, let $|\overline{\mathcal D}_g|$ denote the number of unified types of degenerations of genus $g$. Then 
	\begin{align*}
		|\overline{\mathcal{D}}_g| &< 2^{22g-13}  \cdot e^{\frac{10\pi}{\sqrt{3}}(3g-2)} \cdot (g+1)^{10g-7}\\
		&<(1061g + 1061)^{10g}.
	\end{align*}
	In particular,
	\[
	|\overline{\mathcal D}_g|
	\leq g^{O(g)}
	 \text{ as } g\to\infty.
	\]
\end{theorem}

\begin{proof}
	By Theorem~\ref{Th:Gg}, \ref{Th:AutG} and~\ref{Th:Sf}, we have
	\begin{align*}
		|\overline{\mathcal D}_g|&<|\mathcal{G}_{g}|\cdot|\text{Aut}(G)|\cdot|S(G,\sigma)| \\ 
		& <\frac{81}{4e^{3}(g+1)}(\frac{2e}{3}(g+1))^{3g}\cdot ( \frac{2}{e} )^{3g-3} (3g-2)^{3g-2}\cdot2^{16g-12} e^{\frac{10\pi}{\sqrt{3}}(3g-2)} g^{4g-4} \\ 
		& =2^{22g-17} \cdot 3^{4-3g} \cdot e^{\frac{10\pi}{\sqrt{3}}(3g-2)} \cdot g^{4g-4} \cdot (g+1)^{3g-1} \cdot (3g-2)^{3g-2}\\
		& < 2^{22g-13}  \cdot e^{\frac{10\pi}{\sqrt{3}}(3g-2)} \cdot (g+1)^{10g-7}\\
		& < (1061g + 1061)^{10g}.
	\end{align*}    
\end{proof}

By the construction of Matsumoto-Montesinos (\cite{matsumoto2011pseudo}), the minimal quotient $S_f$ of a pseudo-periodic map $f$ of negative twist is homeomorphic to a normally minimal singular fiber of a degeneration of Riemann surfaces of genus $g$.
We denote by
\[
\mathcal F_g
=
\{[F] \mid F \text{ is a singular fiber  of genus } g\},
\]
where $[F]$ denotes the homeomorphism class of $F$.

The topological equivalence class of the degeneration is determined by the triple $(G, \sigma, S_f)$, while the homeomorphism class of the singular fiber depends only on the minimal generalized quotient space $S_f$. By grouping the generalized quotient spaces that differ only in the value of the integer parameter $K$, we correspondingly obtain the unified types of  singular fibers. 
Consequently, we have the following corollary.

\begin{corollary}
	For any $g\geq 2$, let $|\overline{\mathcal F}_g|$ denote the number of unified types of singular fibers of genus $g$. Then
	\[
	|\overline{\mathcal F}_g|\leq |\overline{\mathcal D}_g|
	\leq  g^{O(g)}
	\text{ as } g\to\infty.
	\]
\end{corollary}

\subsection*{Acknowledgements}
This work is supported by NSFC (No. 12271073).

 \clearpage
\bibliographystyle{amsalpha}
\bibliography{SF}

School of Data Science and Artificial Intelligence, Jilin Engineering Normal University, Changchun, Jilin Province, P. R. of China. 

{\it E-mail address}: \texttt{qiyangkong@163.com}

\vspace{1em}
School of Mathematical Sciences, Dalian University of Technology, Dalian, Liaoning Province, P. R. of China. 

{\it E-mail address}: \texttt{xlliu1124@dlut.edu.cn}

\vspace{1em}
School of Mathematical Sciences, Dalian University of Technology, Dalian, Liaoning Province, P. R. of China.

{\it E-mail address}: \texttt{wangjy0622@mail.dlut.edu.cn}

\end{document}